\documentclass[12pt]{amsart}
\usepackage[T1]{fontenc}
\usepackage[utf8]{inputenc}
\usepackage{fourier}
\usepackage[active]{srcltx}
\usepackage{a4wide}
\usepackage{amsthm,amsfonts,amsmath,mathrsfs,amssymb}
\usepackage{dsfont}
\def\be{\begin{equation}}
\def\ee{\end{equation}}

\newtheorem*{completeness*}{Completeness property}
\newtheorem*{theorem*}{Theorem}
\newtheorem{theorem}{Theorem}

\newtheorem*{proposition*}{Proposition}
\newtheorem{lemma}{Lemma}

\theoremstyle{remark}

\newcommand{\nc}{\newcommand}

\newcommand{\N}{{\mathbb N}}

\nc{\supp}{\operatorname{supp}}
\nc{\Real}{\operatorname{Re}}
\nc{\Imag}{\operatorname{Im}}
\nc{\dif}{\operatorname{d}} \nc{\im}{\operatorname{i}}
\nc{\Hi}{{\mathscr{H}}^\infty} \nc{\Ht}{{\mathscr{H}}^2}
\nc{\Hone}{{\mathscr{H}}^1} \nc{\ol}{\overline} \nc{\bz}{\mathbf{z}}
\nc{\bw}{\mathbf{w}} \nc{\eps}{\varepsilon}

\allowdisplaybreaks[3]

\begin{document}
%\title[GCD sums and extreme values of the Riemann zeta function]{GCD sums and extreme values of the Riemann zeta function}
\title[Extreme values of the Riemann zeta function and its argument]
%An $\Omega_+$ theorem for $S(t)$ and $S_1(t)$ on the Riemann hypothesis]
{Extreme values of the Riemann zeta function and its argument}
%\\ and an $\Omega_+$ theorem for $S(t)$ and $S_1(t)$ on the Riemann hypothesis}
\author{Andriy Bondarenko}
%\address{Department of Mathematical Analysis\\ Taras Shevchenko National University of Kyiv\\
%Volody- myrska 64\\ 01033 Kyiv\\ Ukraine}
\address{Department of Mathematical Sciences \\ Norwegian University of Science and Technology \\ NO-7491 Trondheim \\ Norway}

\email{andriybond@gmail.com}
\author[Kristian Seip]{Kristian Seip}
\address{Department of Mathematical Sciences \\ Norwegian University of Science and Technology \\ NO-7491 Trondheim \\ Norway}
\email{kristian.seip@ntnu.no}
\thanks{Research supported in part by Grant 227768 of the Research Council of Norway. }
\maketitle

\begin{abstract}
We combine our version of the resonance method with certain convolution formulas for $\zeta(s)$ and 
$\log\, \zeta(s)$. This leads to a new $\Omega$ result for $|\zeta(1/2+it)|$: The maximum of $|\zeta(1/2+it)|$ on the interval $1 \le t \le T$ is at least 
$\exp\left((1+o(1)) \sqrt{\log T \log\log\log T/\log\log T}\right)$. We also obtain conditional results for $S(t):=1/\pi$ times the argument of $\zeta(1/2+it)$ and $S_1(t):=\int_0^t S(\tau)d\tau$. On the Riemann hypothesis, the maximum of $|S(t)|$  is at least $c  \sqrt{\log T \log\log\log T/\log\log T}$ and   
the maximum of $S_1(t)$ is at least $c_1  \sqrt{\log T \log\log\log T/(\log\log T)^3}$  on 
the interval $T^{\beta} \le t \le T$ whenever $0\le \beta < 1$. 
% obtain a slight improvement of Tsangs's unconditional $\Omega_+$ result for $S_1(t)$.
\end{abstract}

\section{Introduction}

This paper combines certain convolution formulas for $\zeta(s)$ and $\log\, \zeta(s)$ with the resonance method, as developed in our recent paper \cite{BS1}. As a result, we obtain an improved unconditional $\Omega$ result for $|\zeta(1/2+it)|$ and improved conditional $\Omega$ results for the functions $S(t)$ and $S_1(t)$. We begin by stating a strengthened version of the main theorem of \cite{BS1}.

\begin{theorem} \label{extreme}
Let $0\le \beta<1$ be given and let $c$ be a positive number less than 
$\sqrt{1-\beta}$. If $T$ is sufficiently large, then there exists a $t$, $T^{\beta} \le t \le T$, such that
\begin{equation} \label{eq:ass} \left|\zeta\Big(\frac{1}{2}+it\Big)\right| \ge \exp\left(c\sqrt{\frac{\log T \log\log\log T}{\log\log T}}\right). \end{equation}
\end{theorem}
This implies in particular that $|\zeta(1/2+it)|=\Omega\left(\exp\left((1+o(1))\sqrt{\frac{\log T \log\log\log T}{\log\log T}}\right)\right)$, where the gain compared to \cite{BS1} is that the factor in front of the square-root has been increased from $1/\sqrt{2}+o(1)$ to $1+o(1)$.  We will see that this improvement has the following simple explanation: We avoid using the classical approximation 
\begin{equation} \label{eq:approx}
 \zeta(1/2+i t)= \sum_{n\le T} n^{-1/2-it} - \frac{T^{1/2-it}}{1/2-it}+O(T^{-1/2}),  \quad |t|\le T, \end{equation}
where the second term causes problems when $\beta<1/2$; we replace this approximation by $\zeta(1/2+it)$ convolved with a suitable smooth kernel $K(t)$, so that the influence of the pole of $\zeta(s)$ becomes essentially harmless in the whole range $0\le \beta <1$.

We next turn to our results for the  functions 
\[ S(t):=\frac{1}{\pi} \Imag \log \zeta(1/2+it) \quad \text{and} \quad S_1(t):=\int_0^t S(\tau) d\tau. \] 
Here we use the standard convention that whenever $t$ is not an ordinate of a zero of $\zeta(s)$, $\log \, \zeta(\sigma+it)$ is obtained by continuous variation along the straight line segments joining $2$, $2+it$, and $\sigma+it$, starting from the real value $\zeta(2)$; if $t$ is an ordinate of a zero, then $S$ has a jump discontinuity at $t$, and we declare that $S(t):=
\lim_{\varepsilon\to 0} \big(S(t+\varepsilon)+S(t-\varepsilon)\big)/2$. 
The function $S(t)$ and its primitive $S_1(t)$ 
are instrumental in the study of the finer structure of the vertical distribution of the nontrivial zeros of 
$\zeta(s)$, as seen from the classical Riemann--von Mangoldt formula:
\[ N(t)=\frac{t}{2\pi}\log \frac{t}{2\pi e}+\frac{7}{8}+S(t)+O\left(\frac{1}{t}\right),\]
where as usual $N(t)$ is the number of zeros $\beta+i\gamma$ of $\zeta(s)$ for which $0<\gamma<t$.

The classical bound $S(t)=O(\log t)$ was proved by Backlund, and the fact that $S_1(t)=O(\log t)$ was established by Littlewood. No improvements of these results are known, but on the Riemann hypothesis (RH),  $S(t)=O(\log t/\log\log t)$ and $S_1(t):=O(\log t/(\log\log t)^2)$; see \cite{CCM} for the most recent refinements of these estimates. In the other direction, it is known that
\[ 
S(t)=\begin{cases} \Omega_{\pm}\left((\log  t/\log\log t)^{1/3}\right) &  \\
\Omega_{\pm} \left((\log t/\log\log t)^{1/2}\right) & \text{on RH}\end{cases} \]
by results of respectively Tsang \cite{Ts} and Montgomery \cite{M}, while
\begin{equation} \label{eq:S1b}
S_1(t)=\begin{cases} \Omega_{+}\left((\log t)^{1/2}/(\log\log t)^{-3/2}\right)  \\
\Omega_{-}\left((\log t)^{1/3}\log\log t^{-4/3}\right) \\
\Omega_{-}\left((\log t)^{1/2}(\log\log t)^{-3/2}\right) \quad \text{on RH}
\end{cases} \end{equation}
by work of Tsang, where the first result is from \cite{Ts3} and the two latter bounds are from \cite{Ts}. It is widely believed that the conditional $\Omega$ bounds are closer to the truth. Indeed, a heuristic argument led Montgomery \cite{M} to suggest that the above conditional $\Omega_{\pm}$ bounds for $S(t)$ are optimal; this possible conjecture is also mentioned by Heath-Brown in \cite[p. 384]{T}. Thirty years later, however, Farmer, Gonek, and Hughes  \cite{FGH}  offered alternate heuristic arguments suggesting that one should expect $S(t)$ to grow like $\sqrt{\log t \log\log t}$. Our second theorem shows that at least the conditional $\Omega$ result of Montgomery can be strengthened; we also obtain a conditional $\Omega_+$ result for $S_1(t)$ that supersedes the unconditional one in \eqref{eq:S1b}.

\begin{theorem}\label{thm:Omega} Assume that the Riemann hypothesis is true, and fix $\beta$, $0\le \beta<1$. Then there exist two positive constants $c$ and $c_1$ such that, whenever $T$ is large enough,
\begin{align} \label{eq:Sbound} \max_{T^{\beta}\le t \le T} |S(t)| & \ge c\sqrt{\frac{\log T\log\log\log T}{\log\log T}}, \\ \label{eq:S1bound}
\max_{T^{\beta}\le t \le T} S_1(t) & \ge c_1\sqrt{\frac{\log T\log\log\log T}{(\log\log T)^3}}.
\end{align}
\end{theorem}
This result is of course intimately related with the problem of producing large values of $|\zeta(\sigma+it)|$. We may in fact obtain a slightly weaker result than \eqref{eq:S1bound} as a direct consequence of our earlier work in \cite{BS2}. To this end, we start from the conditional formula (see \cite[(14.12.4)]{T})
%\begin{align*}
\[  \log |\zeta(s)| =\Real i\int_{t/2}^{2t} \frac{S(y)}{s-1/2-iy} dy +O(1). \]
% & = \int_{t/2}^{2t} \frac{S(y)(t-y)}{(\sigma-1/2)^2+(t-y)^2} dy. \end{align*}
We write $s=\sigma+it$ and use integration by parts to obtain
\begin{align*} \Real i\int_{t/2}^{2t} \frac{S(y)}{s-1/2-iy} dy&=\Real \left(i \left[\frac{S_1(y)}{s-1/2-iy}\right]_{t/2}^{2t}
+\int_{t/2}^{2t} \frac{S_1(y)}{(s-1/2-i y)^2}dy\right) \\ 
& = \int_{t/2}^{2t} \frac{S_1(y)((\sigma-1/2)^2-(t-y)^2)}{((\sigma-1/2)^2+(y-t)^2)^2}dy+O(\log t/t^2).\end{align*}
Hence, choosing $\sigma=1/2+1/\log\log t$ and using \cite[Theorem 1]{BS2}, we get conditionally that
\[  \max_{T^{1/2}\le t\le T} |S_1(t)|  \ge c_1 (\log\log T)^{-3/2} \sqrt{\log T \log\log\log T}\]
for all large enough $T$.

The starting point for the proof of Theorem~\ref{thm:Omega} is a convolution formula for $\log\, \zeta(s)$, which was introduced by Selberg \cite{S} and later used also by Tsang \cite{Ts,Ts2} to study $S(t)$ and $S_1(t)$. We will present this formula as well as the corresponding one for $\zeta(s)$ in the next section. On the Riemann hypothesis, the two formulas are very similar, as we will see, and both are in tune with the resonance method, which is a device for picking out large values of Dirichlet series. In Section~\ref{sec:res}, we will present the combinatorial construction from \cite{BS1} underlying our resonators, along with some related estimates to be used in our analysis of the two convolution formulas. The proofs of Theorem~\ref{extreme} and Theorem~\ref{thm:Omega} are then given in respectively Section~\ref{sec:proof1} and Section~\ref{sec:proof2}.

The principal difference between our method and those of for example Selberg and Tsang, is that we use the resonance method rather than high moments to detect large values of Dirichlet series. Also, a principal difference between our version of the resonance method and that used earlier by Soundararajan \cite{So}, is that we use significantly larger primes and a longer Dirichlet polynomial in our resonator. The price we pay compared to any of our predecessors, is that the interval on which increased maxima are known to occur, are considerably larger. In the same vein, we have so far been unable to establish any reasonable estimate for the measure of the set on which corresponding large values are taken, comparable to what was established by Soundararajan in \cite{So}. 

To see the interest of this impasse, we mention without proof that we can adapt Soundararajan's measure result and modify Tsang's proof from \cite{Ts3} to reprove Tsang's unconditional $\Omega_+$ result in \eqref{eq:S1b}. Thus we might hope to establish \eqref{eq:S1bound} unconditionally by proving a stronger measure result than that of \cite{So}, valid for larger values of the Dirichlet series in question.  

We close this introduction by mentioning what is the obstacle for getting improved conditional $\Omega_{\pm}$ results for $S(t)$ and an improved $\Omega_-$ result for $S_1(t)$: Our version of the resonance method only catches large positive values of the real part of a Dirichlet series whose coefficients are all nonnegative. In contrast, when relying on high moments, one is able to catch large values of either sign.
 
\section{Convolution formulas for $\zeta(s)$ and $\log\, \zeta(s)$}
We define the Fourier transform $\widehat{K}$ of $K$ on $\Bbb{R}$ as   
\[ \widehat{K}(\xi):=\int_{-\infty}^{\infty} K(x) e^{-ix\xi} dx. \]
The convolution formula to be used for $\zeta(s)$ should be well known to the experts, but we supply its standard proof for the sake of completeness.
\begin{lemma} \label{lem:zeta}Suppose that $1/2\le \sigma <1$, and let $K(x+iy)$ be an analytic function in the horizontal strip $\sigma-2\le y \le 0$ satisfying the growth estimate
\[ \max_{\sigma-2\le y \le 0} |K(x+iy)|= O\left(\frac{1}{|x|^2}\right) \]
when $|x|\to\infty$. Then for every real $t$ we have
\begin{equation} \label{eq:zeta}
\int_{-\infty}^{\infty} \zeta(\sigma+i (t+u)) K(u) du
=  \sum_{n=1}^\infty \widehat{K}(\log n) n^{-\sigma-it}+2\pi K(-t-i(1-\sigma)).  \end{equation}
\end{lemma}
\begin{proof}
Let $Y$ be a large positive number and $R(Y)$ the rectangle with corners at the points $\sigma\pm iY$ and $2 \pm i Y$. Then by the residue theorem applied to  
the function $f(z):=\zeta(z+it) K(i\sigma-iz)$ in $R(Y)$, we find that 
\[ \int_{-Y}^{Y} \zeta(\sigma+i (t+u)) K(u) du=\int_{-Y}^Y \zeta(2+i(t+u))K(u-i(2-\sigma))du
-2\pi K(t-i(1-\sigma))+O(Y^{-1/2}) \]
by a trivial growth estimate on $\zeta(s)$ when $\Real s \ge 1/2$. Hence
\[ \int_{-\infty}^{\infty} \zeta(\sigma+i (t+u)) K(u) du=\int_{-\infty}^{\infty} \zeta(2+i(t+u))K(u-i(2-\sigma))du
-2\pi K(t-i(1-\sigma)),\]
where both integrals are absolutely convergent by the assumed decay of $K(u)$. Now using the absolutely convergent Dirichlet series of $\zeta(s)$ on the $2$-line and applying Cauchy's theorem termwise to move the integral back to the $\sigma$-line, we reach the desired conclusion.
\end{proof}
The formula of Selberg for $\log\, \zeta(s)$ to be used below, can be found in the following convenient form in \cite[Lemma 5]{Ts}.
\begin{lemma}\label{lem:selberg} Suppose that $1/2\le \sigma <1$, and let $K(x+iy)$ be an analytic function in the horizontal strip $\sigma-2\le y \le 0$ satisfying the growth estimate
\[ V(x):=\max_{\sigma-2\le y \le 0} |K(x+iy)|= O\left(\frac{1}{|x|\log^2 |x|}\right) \]
when $|x|\to\infty$. Then for every $t\neq 0$, we have
\begin{align} \label{eq:Tsang}
\int_{-\infty}^{\infty} \log\, \zeta(\sigma+i (t+u)) K(u) du
= & \sum_{n=2}^\infty \frac{\Lambda(n)}{\log n} \widehat{K}(\log n) n^{-\sigma-it}  \\
\nonumber    & +
2\pi \sum_{\beta>\sigma} \int_0^{\beta-\sigma} 
K(\gamma-t-i \alpha) d\alpha + O(V(t)).  \end{align}
\end{lemma}
Here $\Lambda(n)$ is the classical von Mangoldt function, 
and the second sum is over the zeros $\beta+i\gamma$ of $\zeta(s)$ (if any) satisfying $\beta>\sigma$.  
Thus on the Riemann hypothesis, \eqref{eq:Tsang} reduces to
\begin{equation} \label{eq:Tsang2}
\int_{-\infty}^{\infty} \log\, \zeta(\sigma+i (t+u)) K(u) du
=  \sum_{n=2}^\infty \frac{\Lambda(n)}{\log n} \widehat{K}(\log n) n^{-\sigma-it}   + O(V(t))
\end{equation}
and hence
\begin{equation} \label{eq:Tsang3}
\int_{-\infty}^{\infty} S(t+u) K(u) du
= \frac{1}{\pi} \Imag \sum_{n=2}^\infty \frac{\Lambda(n)}{\log n} \widehat{K}(\log n) n^{-1/2-it}   + O(V(t))
\end{equation}
whenever $K(u)$ is real valued for real arguments $u$.  
Moreover,  using the classical fact that \cite[Theorem 9.9]{T}
\[ h(t):=S_1(t)-\frac{1}{\pi}\int_{1/2}^2 \log|\zeta(\sigma+it)| d\sigma \]
is a bounded function, we infer also from \eqref{eq:Tsang2} Tsang's conditional formula \cite{Ts2}
\begin{equation} \int_{-\infty}^{\infty} \label{eq:Tsang4} (S_1(u+t)-h(u+t)) K(u) du
= \frac{1}{\pi}\Real  \sum_{n=2}^\infty \frac{\Lambda(n)}{\log^2 n} \widehat{K}(\log n)\left( n^{-1/2-it}+O\big(n^{-2}\big)\right)   + O(V(t)), \end{equation}
again assuming that $K(u)$ takes real values for real arguments $u$. 
Notice that here we have extended the definitions of $S(t)$ and $S_1(t)$ in the obvious way so that $S(t)$ is an odd function and $S_1(t)$ is an even function on $\mathbb{R}$. The two conditional formulas \eqref{eq:Tsang3} and \eqref{eq:Tsang4} will be our starting point for the proof of Theorem~\ref{thm:Omega}.

\section{The resonator and associated estimates}\label{sec:res}
A resonator is a function of the form $|R(t)|^2$, where 
\begin{equation} \label{eq:reson} R(t)=\sum_{m\in \mathcal{M}'} r(m) m^{-it}, \end{equation}
and $\mathcal{M}'$ is a suitable finite set of integers. The idea, following \cite{So}, is that $|R(t)|^2$ should ``resonate'' with and pick out large values of the Dirichlet series in question, which will come from the right-hand side of either \eqref{eq:Tsang3} or \eqref{eq:Tsang4}. Before explaining further what this means, we recall the construction of $R(t)$ from \cite{BS1}. 

We begin by fixing a large integer $N$. To simplify the writing, we will  use the short-hand notation $\log_2 x:=\log\log x$ and $\log_3 x:=\log\log\log x$. %, and $\log_4 N:=\log\log\log \log N$.
Let $\gamma$, $0<\gamma<1$, be a parameter to be chosen later, and let $P$ be the set of all primes $p$ such that 
\[ e\log N\log_2 N< p \le \log N\exp( (\log_2 N)^{\gamma})\log_2 N. \]
We define $f(n)$ to be the multiplicative function supported on the set of square-free numbers such that
\[
f(p):=\sqrt{\frac{\log N \log_2N}{\log_3 N}}\frac{1}{\sqrt{p}(\log p-\log_2N-\log_3N)}
\]
for $p$ in $P$ and $f(p)=0$ otherwise. 
%We find that
%\begin{align*}
%\sum_{k, \ell=1}^Nf(n_k)f(n_{\ell})\frac{\gcd(n_k,n_{\ell})}{\sqrt{n_k
%n_{\ell}}}& \ge\sum_{k=1}^N\frac{f(n_k)}{\sqrt{n_k}}\sum_{n_{\ell}|n_k}f(n_{\ell})\frac{\gcd(n_k,n_{\ell})}
%{\sqrt{n_{\ell}}} \\
%& =\sum_{k=1}^N\frac{f(n_k)}{\sqrt{n_k}}\sum_{n_{\ell}|n_k}f(n_{\ell})\sqrt{n_{\ell}}. \end{align*}

Let $P_k$ be the set of all primes $p$ such that $e^k\log N\log_2N<p\le e^{k+1}\log N\log_2N$ for $k=1,\ldots,[(\log_2 N)^{\gamma}]$.
Fix $1 < a < 1/\gamma$. Then let $M_k$ be the set of those integers having at least $\frac{a\log N}{k^2\log_3N}$ prime divisors in $P_k$,
and let $M'_k$ be the set of integers from $M_k$ that have prime divisors only in $P_k$.
Finally, set 
\[ \mathcal{M}:=\supp(f)\setminus\bigcup_{k=1}^{[(\log_2N)^{\gamma}]}M_k.\] In other words,
$\mathcal{M}$ is the set of square-free numbers $n$ that have at most $\frac{a\log N}{k^2\log_3N}$
divisors in each group $P_k$. It is is clear that $\mathcal{M}$ is divisor closed, by which we mean that
$d$ is in $\mathcal{M}$ whenever $m$ is in $\mathcal{M}$ and $d$ divides $m$.

The first of the following two lemmas was established as part of the proof of \cite[Lemma~2]{BS1}.
\begin{lemma}\label{lem:card}
We have $|\mathcal{M}|\le N$ whenever $N$ is large enough, depending on $a$ and $\gamma$.
\end{lemma}
\begin{lemma}
\label{lem:new3}
We have
\be
\label{ll2}
\frac{1}{\sum_{i\in\N}f(i)^2}\sum_{n\in\mathcal{M}}f(n)^2\sum_{p|n}\frac{1}{f(p)\sqrt{p}}  \ge(\gamma+o(1))\sqrt{\frac{\log N \log_3 N}{\log_2 N}}.
\ee
\end{lemma}
\begin{proof}
The proof is similar to that of~\cite[Lemma 2]{BS1}. Fix $\alpha$ such that $0<\alpha<1$.
Let $L_k$ be the set of integers in $\supp(f)$ that have at most $\frac{\alpha\log N}{k^2\log_3N}$ prime divisors in $P_k$,
and let $L'_k$ be the set of integers from $L_k$ that have prime divisors only in $P_k$.
Finally, set
\[ \mathcal{L}:=\mathcal{M}\setminus\bigcup_{k=1}^{(\log_2N)^{\gamma}}L_k.\] In other words,
$\mathcal{L}$ is the set of numbers in $\mathcal{M}$ that have at least $\frac{\alpha\log N}{k^2\log_3N}$
divisors in each group $P_k$. To prove the lemma, it is enough to show that
\be
\label{ll3}
\frac{1}{\sum_{i\in\N}f(i)^2}\sum_{n\not\in\mathcal{L}}f(n)^2=o(1), \qquad N\to\infty.
\ee
Indeed, \eqref{ll3} implies that the left-hand side of~\eqref{ll2} is at least
\begin{align*}
(1-o(1)) & \min_{n\in\mathcal{L}}\sum_{p|n}\frac{1}{f(p)\sqrt{p}}
 \ge(1-o(1))\sum_{k=1}^{(\log_2N)^{\gamma}}\frac{\alpha\log N}{k^2\log_3N}\min_{p\in P_k}\frac{1}{f(p)\sqrt{p} } \\
& \ge(1-o(1))\sum_{k=1}^{(\log_2N)^{\gamma}}\frac{\alpha\log N}{k^2\log_3N}k\sqrt{\frac{\log_3N }{\log N\log_2 N}}  \ge (1-o(1))\alpha\gamma\sqrt{\frac{\log N \log_3 N}{\log_2 N}},
\end{align*}
which implies the statement of the lemma, if we choose $\alpha$ arbitrarily close to 1.

We turn to the proof of~\eqref{ll3}. Since
\[ \mathcal{L}=\supp(f)\setminus\bigcup_{k=1}^{(\log_2N)^{\gamma}}\left(M_k\cup L_k\right),\]
it is enough to prove that
\be
\label{ll4}
\frac{1}{\sum_{i\in\N}f(i)^2}\sum_{k=1}^{(\log_2N)^{\gamma}}\sum_{n\in L_k}f(n)^2=o(1),
\ee
and
\be
\label{ll5}
\frac{1}{\sum_{i\in\N}f(i)^2}\sum_{k=1}^{(\log_2N)^{\gamma}}\sum_{n\in M_k}f(n)^2=o(1).
\ee
We will only prove~\eqref{ll4}; the proof of~\eqref{ll5} is similar and was essentially done in~\cite[Lemma 2]{BS1}.

For every fixed $k$, we have
\[
\frac{1}{\sum_{i\in\N}f(i)^2}\sum_{n\in L_k}f(n)^2=\frac{1}{\prod_{p\in P_k}{(1+f(p)^2)}}\sum_{n\in L'_k}f(n)^2.
\]
Using that $f(n)$ is multiplicative and the definition of $L'_k$, we find that
\[
\sum_{n\in L'_k}f(n)^2\le b^{-\alpha\frac{\log N}{k^2\log_3N}}\prod_{p\in P_k}(1+bf(p)^2)
\]
for a suitable $b<1$, and hence
\begin{equation}
\label{ll6}
\frac{1}{\prod_{p\in P_k}(1+f(p)^2)}\sum_{n\in L'_k}f(n)^2\le b^{-\alpha\frac{\log N}{k^2\log_3N}}\exp\left(\sum_{p\in P_k}(b-1) f(p)^2\right).
\end{equation}
We now recall from~\cite[Lemma 2]{BS1} that
\[
\sum_{p\in P_k} f(p)^2\le (1+o(1))\frac{\log N}{k^2\log_3N}.
\]
Therefore the right-hand side  of~\eqref{ll6} is at most
\[
\exp\left((b-1-\alpha \log b +o(1))\frac{\log N}{k^2\log_3N}\right).
\]
Choosing $b$ sufficiently close to 1, we obtain $b-1-\alpha\log b <0$. This gives~\eqref{ll4} and hence~\eqref{ll3}.
\end{proof}

We proceed as in \cite{BS1} (following an idea from \cite{A}) and let $\mathcal{J}$ be the set of integers $j$ such that
\[ \Big[(1+T^{-1})^j,(1+T^{-1})^{j+1}\Big)\bigcap \mathcal{M} \neq \emptyset,  \]
and we let $m_j$ be the minimum of  $\big[(1+T^{-1})^j,(1+T^{-1})^{j+1}\big)\bigcap \mathcal{M}$ for $j$ in $\mathcal{J}$. We then set
\[ \mathcal{M}':= \big \{ m_j: \ j\in \mathcal{J} \big\}\]
and 
\[ r(m_j):= \left(\sum_{n\in \mathcal{M}, (1-T^{-1})^{j-1} \le n \le (1+T^{-1})^{j+2}} f(n)^2\right)^{1/2} \] 
for every $m_j$ in $\mathcal{M}'$. This defines the resonator \eqref{eq:reson}; note that plainly $|\mathcal{M}'|\le |\mathcal{M}| \le N$. 

We will in what follows, for a reason that will become clear later, require that $N= [T^{\kappa}]$ for some $\kappa$, $0<\kappa\le 1$. Also, as in \cite{BS1}, we set $\Phi(t):=e^{-t^2/2}$. We now turn to an estimation of  three integrals involving $|R(t)|^2\Phi(t/T)$ that will be essential in the proofs of our two theorems.  
\begin{lemma} \label{lem:m1} We have
\begin{equation} \label{m1f} \int_{-\infty}^\infty |R(t)|^2 \Phi\Big(\frac{t}{T} \Big) dt  \ll T  \sum_{n\in \mathcal{M}} f(n)^2. \end{equation} 
\end{lemma}
\begin{proof} We begin by noting that
\begin{equation}\label{m1}  \int_{-\infty}^\infty |R(t)|^2 \Phi\Big(\frac{t}{T} \Big) dt=\sqrt{2\pi} T  \sum_{m,n\in \mathcal{M}'} r(m)r(n) 
\Phi\Big(T \log \frac{m}{n} \Big) \end{equation}
since
\[ \widehat{\Phi}(x)=\int_{-\infty}^\infty \Phi(t)e^{-itx} dt= \sqrt{2\pi} \Phi(x).  \]
Using the definition of $\mathcal{M}'$, we find that
\begin{equation} \label{eq:rf}
\sum_{m\in \mathcal{M}'} r(m)^2 \le 3 \sum_{n\in \mathcal{M}} f(n)^2.
\end{equation}
To deal with the off-diagonal terms, we find, using again the definition of  $\mathcal{M}'$, that  
\begin{align} \label{eq:off} \sum_{m,n\in \mathcal{M}', m\neq n} r(m)r(n) 
\Phi\Big(T \log \frac{m}{n} \Big) & \le   \sum_{j,\ell\in \mathcal{J}, j\neq \ell} r(m_j)r(n_{\ell}) 
\Phi\left(T (|j-\ell|-1)\log(1+T^{-1})\right) \\ \nonumber 
& \ll \sum_{j,\ell\in \mathcal{J}, j\neq \ell} r(m_j)r(n_{\ell}) 
\Phi\left(|j-\ell |-1\right) \\ \nonumber 
& \ll \sum_{j,\ell\in \mathcal{J}, j\neq \ell} r(m_j)^2
\Phi\left(|j-\ell |-1\right) \ll \sum_{m\in \mathcal{M}'} r(m)^2.
\end{align}
Here we used the Cauchy--Schwarz inequality, the definition of $r(m)$, and finally the rapid decay of $\Phi(t)$. Plugging \eqref{eq:rf} and \eqref{eq:off} into \eqref{m1}, we arrive at \eqref{m1f}.
\end{proof}

The proofs of the next two lemmas follow closely an argument that may be found in \cite[p. 1699]{BS1}. Here an essential role is played by the following way of relating certain sums of coefficients over the two sets $\mathcal{M}$ and $\mathcal{M}'$. For a given $k$ in $\mathcal{M}$,  consider all pairs $m',n'$ in $\mathcal{M}'$ such that $ |km'/n' -1|\le 3/T$. We use the notation 
\[ J(m'):=\left[(1+T^{-1})^j,(1+T^{-1})^{j+1}\right), \]
where $j$ is the unique integer such that 
$(1+T^{-1})^j\le m'<(1+T^{-1})^{j+1}$. Using the Cauchy--Schwarz inequality and the definition of $r(m')$, we find that
\[ \sum_{m,n\in \mathcal{M}, mk=n, m\in J(m'),  n\in J(n')}f(m)f(n)\le r(m')r(n') \]
and hence, by the definition of $\mathcal{M}'$, that
\begin{equation}\label{eq:base} \sum_{m,n\in \mathcal{M}, mk=n}f(m)f(n)\le \sum_{m',n'\in \mathcal{M}', |km'/n' -1|\le 3/T} r(m')r(n'). \end{equation}
\begin{lemma}\label{lem:m2}
Suppose that
\begin{equation} \label{eq:gen}  F(t):=\sum_{n=1}^\infty  a_n  n^{-1/2-it} \end{equation}
is absolutely convergent and that $a_n\ge 0$ for every $n$. Let $\varepsilon$ be a positive number and $\gamma$ be the parameter defining the set $P$. Then
\begin{equation}\label{m2f} \int_{-\infty}^{\infty} F(t) |R(t)|^2 \Phi\Big(\frac{t}{T} \Big) dt  \ge T\left( \min_{n\le T^{\varepsilon}} a_n\right) \exp\left(\big(\gamma+o(1)\big)\sqrt{\kappa \frac{\log T \log_3 T}{\log_2 T}}\right) \sum_{n\in \mathcal{M}}f(n)^2\end{equation}
when $T\to\infty$, where the function $o(1)$ depends on the parameters $\gamma$, $\kappa$, and $\varepsilon$, but not on $F$.
\end{lemma}
\begin{proof} We use the explicit expression for $R(t)$ and integrate termwise to get
\begin{align*}  \int_{-\infty}^{\infty} F(t) |R(t)|^2 \Phi\Big(\frac{t}{T} \Big) dt & =\sqrt{2\pi} T\sum_{m,n\in \mathcal{M}'} \sum_{k=1}^{\infty} \frac{ a_k r(m)r(n)}{\sqrt{k}} \Phi\Big(T \log \frac{km}{n} \Big) \\
& \ge \sqrt{2\pi} T \left(\min_{j\le T^{\varepsilon} } a_j \right) \sum_{m,n\in \mathcal{M}'} \sum_{k\le T^{\varepsilon}} \frac{r(m)r(n)}{\sqrt{k}} \Phi\Big(T \log \frac{km}{n}\Big). \end{align*}
In the last step, we used that all the terms in the series are positive, so that we could sum over a suitable finite subcollection of them. As in \cite[p. 1699]{BS1}, we change the order of summation and sum only over those $m$ and $n$ such that \eqref{eq:base} applies; the remaining part of the proof is identical to the estimation of the quantity $I(R,T)$ in \cite[p. 1699]{BS1}, leading to the displayed formula (25) in \cite{BS1}. We therefore omit the details. \end{proof}
\begin{lemma}\label{lem:m2log}
There exists a positive constant $c$ such that if   
\[ G(t):=\sum_{n=2}^\infty \frac{\Lambda(n) a_n}{\log n}  n^{-1/2-it} \]
is absolutely convergent and $a_n\ge 0$ for every $n$, then  
\[ \int_{-\infty}^{\infty} G(t) |R(t)|^2 \Phi\Big(\frac{t}{T} \Big) dt \ge c T 
\sqrt{\frac{\log T\log_3 T}{\log_2 T}} \left(\min_{p\in P} a_p\right) \sum_{n\in \mathcal{M}} f(n)^2. \]
\end{lemma}

\begin{proof} We use again the explicit expression for $R(t)$ and integrate termwise. This gives 
\begin{align*}  \int_{-\infty}^{\infty} G(t) |R(t)|^2 \Phi\Big(\frac{t}{T} \Big) dt & =\sqrt{2\pi} T\sum_{m,n\in \mathcal{M}'} \sum_{k=2}^{\infty} \frac{\Lambda(k) a_k r(m)r(n)}{\sqrt{k}\log k} \Phi\Big(T \log \frac{km}{n} \Big) \\
& \ge \sqrt{2\pi} T \left(\min_{p\in P} a_p\right) \sum_{m,n\in \mathcal{M}'} \sum_{p\in P} \frac{r(m)r(n)}{\sqrt{p}} \Phi\Big(T \log \frac{pm}{n}\Big), \end{align*}
%By the Cauchy--Schwarz inequality, we have
%\[ \sum_{m,n\in \mathcal{M}'} \sum_{k\in \mathcal{M}, k> T} \frac{r(m)r(n)}{\sqrt{k}}
%\le T^{-1/2} |\mathcal{M}|\cdot |\mathcal{M}'| \ \sum_{m\in \mathcal{M}'} r(m)^2, \]
%and hence we obtain, using that $|\mathcal{M}'|\le |\mathcal{M}|\le N\le T^{\kappa}$ and the assumption that $\kappa<1/4$,  that
%\begin{equation} \label{ii} I(R,T)=\frac{\sqrt{2\pi} T}{\log T} \sum_{m,n\in \mathcal{M}'} \sum_{k\in \mathcal{M}} \frac{r(m)r(n)}{\sqrt{k}} \Phi
%\Big(\frac{T}{\log T} \log \frac{km}{n} \Big) + o(T) \sum_{n\in \mathcal{M}}f(n)^2. \end{equation}
where we as in the preceding case used that all the terms in the series are positive. We sum again over those indices $m$ and $n$ in $\mathcal{M}$ such that \eqref{eq:base} applies. This means that if we set $k=p$, divide \eqref{eq:base} by $\sqrt{p}$, and sum over all the primes $p$ in $P$, then we get
\[ \label{last} \int_{-\infty}^{\infty} G(t) |R(t)|^2 \Phi\Big(\frac{t}{T} \Big) dt \gg T %\frac{1}{\sum_{i\in\N}f(i)^2}
\left(\min_{p\in P} a_p\right)\sum_{n\in\mathcal{M}}f(n)^2\sum_{p|n}\frac{1}{f(p)\sqrt{p}}. \]
Hence, using Lemma~\ref{lem:new3}, we see that
\begin{equation} \label{m2f}  \int_{-\infty}^{\infty} G(t) |R(t)|^2 \Phi\Big(\frac{t}{T} \Big) dt\gg T \sqrt{\frac{\log T \log_3 T}{\log_2 T}}\left( \min_{p\in P} a_p\right)\sum_{n\in \mathcal{M}} f(n)^2.\end{equation}
\end{proof}

\section{Proof of Theorem~\ref{extreme}}\label{sec:proof1}

The proof of Theorem~\ref{extreme} is a little easier than the proof of the main theorem of \cite{BS1}, because our convolution formula allows us to estimate more crudely. We choose
\[ K(t):=\frac{\sin^2((\varepsilon \log T)t)}{(\varepsilon \log T)t^2} \]
where $\varepsilon>0$ can be chosen as small as we please. We notice that 
\begin{equation}\label{eq:fourier1} \widehat{K}(\xi)=\frac{\pi}{2} \max\left(\left(1-\frac{|\xi |}{2 \varepsilon \log T}\right), 0\right). \end{equation} 
We find that
\begin{align*} \int_{-T^{\beta}}^{T^\beta} \int_{-\infty}^{\infty} |\zeta(1/2+i (t+u))| K(u) dudt
& \ll T^{\beta} + \int_{-T^{\beta}}^{T^\beta} \int_{|u|\le T^\beta} |\zeta(1/2+i (t+u))| K(u) dudt \\
& \ll T^{\beta} + \int_{-2T^{\beta}}^{2T^\beta} |\zeta(1/2+i t)| dt  \ll T^\beta \sqrt{\log T}, \end{align*}
where we in the last step used the Cauchy--Schwarz inequality and Hardy and Littlewood's classical bound for the second moment of $\zeta(1/2+it)$. Hence
\begin{align*} \int_{-T^{\beta}}^{T^\beta} \int_{-\infty}^{\infty} |\zeta(1/2+i (t+u))| K(u) du |R(t)|^2 \Phi\left(\frac{t}{T}\right) dt
& \ll T^\beta \sqrt{\log T} R(0)^2 \\ 
& \ll T^{\beta+\kappa} \sqrt{\log T}\sum_{n\in \mathcal{M}}f(n)^2 \end{align*} 
by a trivial estimation of $R(0)^2$. Plainly, by the rapid decay of $\Phi(t)$, we also have
\[  \int_{|t|>T\log T} \int_{-\infty}^{\infty} |\zeta(1/2+i (t+u))| K(u) du |R(t)|^2 \Phi\left(\frac{t}{T}\right) dt
=o(1) \sum_{n\in \mathcal{M}}f(n)^2,\]
whence
\begin{align} \label{eq:conv1}
& \int_{T^{\beta}\le |t| \le T\log T} \int_{-\infty}^{\infty}  \zeta(1/2+i (t+u)) K(u) du |R(t)|^2 \Phi\left(\frac{t}{T}\right) dt \\
&\nonumber  \quad = \int_{-\infty}^{\infty} \int_{-\infty}^{\infty} \zeta(1/2+i (t+u)) K(u) du |R(t)|^2 \Phi\left(\frac{t}{T}\right) dt + O\Big(T^{\beta+\kappa}\sqrt{\log T}\Big) \sum_{n\in \mathcal{M}}f(n)^2.
\end{align}
We now require that $\kappa<1-\beta$ and see that by applying Lemma~\ref{lem:m1} to the left-hand side of \eqref{eq:conv1}, we obtain
\begin{align} \label{eq:conv2}
& \max_{T^{\beta}/2\le t \le 2T\log T} |\zeta(1/2+it)| T \sum_{n\in \mathcal{M}}f(n)^2 \\ \nonumber
& \quad \gg \int_{-\infty}^{\infty} \int_{-\infty}^{\infty} \zeta(1/2+i (t+u)) K(u) du |R(t)|^2 \Phi\left(\frac{t}{T}\right) dt + O(T) \sum_{n\in \mathcal{M}}f(n)^2.
\end{align} 
We now set
\[ F(t):=\sum_{n=1}^\infty \widehat{K}(\log n) n^{-1/2-it} \] and see that by Lemma~\ref{lem:zeta}, the double integral on the right-hand side of \eqref{eq:conv2} takes the form
\begin{align*} \int_{-\infty}^{\infty}  \int_{-\infty}^{\infty} & \zeta(1/2+i (t+u))  K(u) du |R(t)|^2 \Phi\left(\frac{t}{T}\right) dt \\ & =\int_{-\infty}^{\infty} F(t) |R(t)|^2 \Phi\Big(\frac{t}{T} \Big) dt +2\pi \int_{-\infty}^{\infty} K(t-i/2) |R(t)|^2 \Phi\Big(\frac{t}{T} \Big) dt.
  \end{align*}
We invoke  Lemma~\ref{lem:m2} to estimate the first term on the right-hand side, and we estimate the second term by using the explicit expression for $K(t-i/2)$ and using again the trivial estimate $|R(t)|\le R(0)$. This gives us
\begin{align} \label{eq:fin} \int_{-\infty}^{\infty}  \int_{-\infty}^{\infty} & \zeta(1/2+i (t+u))  K(u) du |R(t)|^2 \Phi\left(\frac{t}{T}\right) dt \\ \nonumber & \gg \left(T \left( \min_{n\le T^{\varepsilon}} \widehat{K}(\log n)\right) \exp\left(\big(\gamma+o(1)\big)\sqrt{\kappa \frac{\log T \log_3 T}{\log_2 T}}\right) +O(T^{\kappa+\varepsilon})\right) \sum_{n\in \mathcal{M}}f(n)^2.
  \end{align}
In view of \eqref{eq:fourier1}, $\min_{n\le T^{\varepsilon}} \widehat{K}(\log n)$ is bounded below by $\pi/4$. Hence, choosing $\varepsilon$ small enough and plugging \eqref{eq:fin} into \eqref{eq:conv2}, we find that
the asserted bound \eqref{eq:ass} holds for some $t$ satisfying $T^{\beta}/2 \le t \le 2 T \log T$. We obtain the desired restriction  $T^{\beta} \le t \le T$ after a trivial adjustment, changing $T$ to $T/(2\log T)$ and making $\beta$ slightly smaller. 

\section{Proof of Theorem~\ref{thm:Omega}}\label{sec:proof2}
The proof of Theorem~\ref{thm:Omega} is very similar to the preceding proof, but there is an interesting distinction: In Lemma~\ref{lem:selberg}, we choose $K(t)$  to be an odd function when dealing with $S(t)$ and an even function when dealing with $S_1(t)$; this difference is the reason why we only obtain a conditional  $\Omega$ result for $S(t)$.

\begin{proof}[Proof of \eqref{eq:Sbound}]
We now choose
\[ K(t):=-(\log_2 T)^2 t \Phi((\log_2 T)t), \]
which has Fourier transform
\begin{equation}\label{eq:fourier} \widehat{K}(\xi)= i \sqrt{2\pi} (\log_2 T)^{-1} \xi \Phi(\xi/\log_2 T). \end{equation}
We compute in the same fashion as above:
\begin{align*} \int_{-T^{\beta}}^{T^\beta} \int_{-\infty}^{\infty} |S(t+u)) K(u)| dudt
& \ll T^{\beta} + \int_{-T^{\beta}}^{T^\beta} \int_{|u|\le T^\beta} |S(t+u) K(u)| dudt \\
& \ll T^{\beta} + \int_{-2T^{\beta}}^{2T^\beta} |S(t)| dt  \ll T^\beta \sqrt{\log_2T}, \end{align*}
where we in the last step used the Cauchy--Schwarz inequality and a classical bound of Selberg \cite{S}  for the second moment of $S(t)$. We follow the same steps as in the preceding case and hence, requiring that $\kappa<1-\beta$, we find that 
\begin{align} \label{eq:conv3}
 \int_{T^{\beta}\le |t| \le T\log T} \int_{-\infty}^{\infty} &  S(t+u) K(u) du |R(t)|^2 \Phi\left(\frac{t}{T}\right) dt \\
&\nonumber   = \int_{-\infty}^{\infty} \int_{-\infty}^{\infty} S(t+u)) K(u) du |R(t)|^2 \Phi\left(\frac{t}{T}\right) dt + O(T) \sum_{n\in \mathcal{M}}f(n)^2.
\end{align}  Applying Lemma~\ref{lem:m1} to the left-hand side of \eqref{eq:conv3}, we obtain
\begin{align} \label{eq:conv4}
& \max_{T^{\beta}/2\le t \le 2T\log T} |S(t)| T \sum_{n\in \mathcal{M}}f(n)^2 \\ \nonumber
& \quad \gg \int_{-\infty}^{\infty} \int_{-\infty}^{\infty} S(t+u) K(u) du |R(t)|^2 \Phi\left(\frac{t}{T}\right) dt + O(T) \sum_{n\in \mathcal{M}}f(n)^2.
\end{align} 
We now set
\[ G(t):=\sum_{n=2}^\infty \frac{\Lambda(n) \widehat{K}(\log n)}{\pi \log n} n^{-1/2-it} \] and see that by \eqref{eq:Tsang3}, the double integral on the right-hand side of \eqref{eq:conv4} takes the form
\begin{align*} \int_{-\infty}^{\infty}  \int_{-\infty}^{\infty} & S(t+u)  K(u) du |R(t)|^2 \Phi\left(\frac{t}{T}\right) dt \\ & =\Imag \int_{-\infty}^{\infty} G(t) |R(t)|^2 \Phi\Big(\frac{t}{T} \Big) dt +O\left(\int_{-\infty}^{\infty} V(t) |R(t)|^2 \Phi\Big(\frac{t}{T} \Big) dt\right).
  \end{align*}
We invoke  Lemma~\ref{lem:m2log} to estimate the first term on the right-hand side, and we estimate the second term by using the explicit expression for $V(t)$ and using again the trivial estimate $|R(t)|\le R(0)$. This gives us
\begin{align} \label{eq:fin2} \int_{-\infty}^{\infty}  \int_{-\infty}^{\infty} & S (t+u)  K(u) du |R(t)|^2 \Phi\left(\frac{t}{T}\right) dt \\ \nonumber & \gg \left(T \left( \min_{p\in P} \Imag \widehat{K}(\log p)\right) \sqrt{ \frac{\log T \log_3 T}{\log_2 T}} +O(T^{\kappa+\varepsilon})\right) \sum_{n\in \mathcal{M}}f(n)^2
  \end{align}
for every $\varepsilon>0$. By \eqref{eq:fourier} and the definition of $\mathcal{M}$, 
\[  \min_{p\in P} \Imag \widehat{K}(\log p)=\sqrt{2\pi} \min_{e\log N \log_2 N\le p \le e^{(\log_2 N)^{\gamma}}\log N \log_2 N}\frac{\log p\ \Phi(\log p/\log_2 T)}{\log_2 T} \gg 1\] 
since $N=[T^{\kappa}]$.
Choosing $\varepsilon$ small enough and plugging \eqref{eq:fin2} into \eqref{eq:conv4}, we therefore obtain the asserted bound \eqref{eq:Sbound}, after the same trivial adjustment of $T$ and $\beta$ as in the preceding case. 
\end{proof}

\begin{proof}[Proof of \eqref{eq:S1bound}]
We now choose
\[ K(t):=\log_2 T  \Phi((\log_2 T)t), \]
which has Fourier transform
\[ \widehat{K}(\xi)=  \sqrt{2\pi} \Phi(\xi/\log_2 T). \]
Computing exactly as in the preceding case, we obtain
\begin{align} \label{eq:conv5}
& \max_{T^{\beta}/2\le t \le 2T\log T} S_1(t) T \sum_{n\in \mathcal{M}}f(n)^2 \\ \nonumber
& \quad \gg \int_{-\infty}^{\infty} \int_{-\infty}^{\infty} S_1(t+u) K(u) du |R(t)|^2 \Phi\left(\frac{t}{T}\right) dt + O(T) \sum_{n\in \mathcal{M}}f(n)^2.
\end{align} 
We now set
\[ G(t):=\sum_{n=2}^\infty \frac{\Lambda(n) \widehat{K}(\log n)}{\pi \log^2 n} n^{-1/2-it} \] and see that by \eqref{eq:Tsang4}, the double integral on the right-hand side of \eqref{eq:conv5} takes the form
\begin{align*} \int_{-\infty}^{\infty}  \int_{-\infty}^{\infty} & S_1(t+u)  K(u) du |R(t)|^2 \Phi\left(\frac{t}{T}\right) dt \\ & =\Real \int_{-\infty}^{\infty} G(t) |R(t)|^2 \Phi\Big(\frac{t}{T} \Big) dt +O\left(\int_{-\infty}^{\infty} \left(V(t)+1\right) |R(t)|^2 \Phi\Big(\frac{t}{T} \Big) dt\right),
  \end{align*}
where $V(t)$ is the function introduced in Lemma~\ref{lem:selberg}. We invoke  Lemma~\ref{lem:m2log} again to estimate the first term on the right-hand side, and we estimate the second term by using the explicit expression for $V(t)$  and using again the trivial estimate $|R(t)|\le R(0)$. This gives us
\begin{align} \label{eq:fin3} \int_{-\infty}^{\infty}  \int_{-\infty}^{\infty} & S_1 (t+u)  K(u) du |R(t)|^2 \Phi\left(\frac{t}{T}\right) dt \\ \nonumber & \gg \left(T \left( \min_{p\in P} \frac{\widehat{K}(\log p)}{\log p}\right) \sqrt{ \frac{\log T \log_3 T}{\log_2 T}} +O(T^{\kappa+\varepsilon})+O(T)\right) \sum_{n\in \mathcal{M}}f(n)^2
  \end{align}
 for every $\varepsilon>0$.
Plainly,
\[  \min_{p\in P} \frac{\widehat{K}(\log p)}{\log p}\gg \frac{1}{\log_2 T}, \]
and hence choosing $\varepsilon$ small enough and plugging \eqref{eq:fin3} into \eqref{eq:conv5}, we obtain the asserted bound \eqref{eq:S1bound}, again adjusting $T$ and $\beta$ appropriately.
\end{proof} 

\section*{Acknowledgements}

We would like to thank Dennis Hejhal for an inspiring correspondence which led us to carry out the research presented in this paper. We are also indebted to Daniel Goldston for a pertinent bibliographical remark on the first version of this paper. Finally, we would like to express our gratitude to the referee for a very careful review.

\end{document}